\documentclass[11pt,reqno]{amsart}
\usepackage[titletoc,title]{appendix}
\usepackage{amsmath,amsthm,amsfonts,amssymb,mathrsfs,epsfig,bm,bbm}
\usepackage[usenames]{color}
\usepackage[colorlinks=true, urlcolor={black}]{hyperref}
\usepackage{dsfont}
\usepackage{mathtools}
\mathtoolsset{showonlyrefs=true}
\usepackage{enumerate}
\usepackage{geometry}
\usepackage{marginnote}

\newtheorem{theorem}{Theorem}[section]
\newtheorem{lemma}[theorem]{Lemma}
\newtheorem{corollary}[theorem]{Corollary}
\newtheorem{proposition}[theorem]{Proposition}

\theoremstyle{remark}

\newtheorem{remark}{Remark}[section]

\def\Z{\mathbb{Z}}

\def\R{\mathbb{R}}

\def\P{\mathbb{P}}

\renewcommand{\phi}{\varphi}
\renewcommand{\epsilon}{\varepsilon}

\renewcommand{\limsup}{\varlimsup}
\renewcommand{\liminf}{\varliminf}

\DeclareMathOperator{\sign}{sign}

\DeclareMathOperator{\argmin}{arg\ min}

\renewcommand{\d}{\text{\rm\,d}}

\newcommand{\cM}{\mathcal{M}}

\definecolor{mygray}{gray}{0.9}

\newcommand{\keyword}[1]{ \noindent {\footnotesize
             {\small \em Keywords and phrases.} {\sc #1} } }
\newcommand{\ams}[2]{  \noindent {\footnotesize
             {\small \em AMS {\rm 2000} subject classifications.
             {\rm {\sc #1}; {\sc #2}} } } }

 \title{The greedy walk on an inhomogeneous Poisson process}
 \author{Katja~Gabrysch and Erik~Th\"{o}rnblad}
 \date{\today}
 \address{Department of Mathematics, Uppsala University, Box 480, S-75106 Uppsala, Sweden.}
 \email{katja.gabrysch@math.uu.se, erik.thornblad@math.uu.se}
\begin{document}

\begin{abstract}
 The greedy walk is a deterministic walk that always moves from its current position to the nearest not yet visited point.
 In this paper we consider the greedy walk on an inhomogeneous Poisson point process on the real line. Our primary interest is whether the walk visits all points of the point process, and we determine sufficient and necessary conditions on the mean measure of the point process for this to happen. Moreover, we provide precise results on threshold functions for the property of visiting all points.

\vspace{2mm} 

\keyword{greedy walk; inhomogeneous Poisson point processes; threshold}

\ams{60K37}{60G55,60K25}

\end{abstract}

 \maketitle

\section{Introduction and main results}

Consider a simple point process $\Pi$ without accumulation points in a metric space $(E, d)$. We think of $\Pi$ either as
an integer-valued  measure or as a collection of points (the support of the measure). With the latter viewpoint in mind, we define the greedy walk on $\Pi$ as follows. 
Let $S_0\in E$ and $\Pi_0=\Pi$. 
Define, for $n\geq 0$,
\begin{align}
 S_{n+1} &= \argmin \{d(S_n,X)  :  X\in \Pi_n\}, \\
  \Pi_{n+1} &= \Pi_n\setminus \{S_{n+1}\}.
\end{align}
The set $\Pi_n$ denotes the set of unvisited points of $\Pi$ up until (and including) time $n$. We write $\Pi_{\infty}=\bigcap_{n=1}^{\infty}\Pi_n$ for the set of points that are never visited by the walk. Once the underlying environment $\Pi$ is fixed, the process $(S_n)_{n=0}^{\infty}$ is deterministic 
(except possibly for ties which need to be broken, but these will almost surely not occur in our setting).

The greedy walk has been studied before in the literature, with various choices of the underlying point process.
When $\Pi$ is a homogenous Poisson process on $\R$, one can show using a Borel--Cantelli--type argument that
$\Pi_{\infty}\neq \emptyset$ with probability $1$, that is, the greedy walk does not visit all points of 
the underlying point process. 
More precisely, the expected number of times the greedy walk starting from $0$ changes sign is $1/2$ \cite{Ga}.
Due to this, Rolla {\it et al.}\ \cite{RST14} considered a related problem, in which each point in the process can be visited either once, with probability $1-p$, or twice, with probability $p$. They show that $\Pi_{\infty}=\emptyset$, for any $0<p<1$, meaning that every point is eventually visited. 
Another modification of the greedy walk on $\R$ is studied by Foss {\it et al.}\ \cite{FRS}. 
The authors considered a dynamic version of the greedy walk, 
where the times and positions of new points arriving in the 
system are given by a Poisson process on the space-time half-plane.
They show that the greedy walk still diverges to infinity in one direction and does not visit all points.
In the survey paper \cite{BFL}, Bordenave {\it et al.} state several questions about the behaviour of the greedy walk on an inhomogeneous Poisson process in $\R^d$. We resolve here the problem for $d=1$. 

In this paper we define $\Pi$ to be an inhomogeneous Poisson process on $\R$ (with the Euclidean metric) given by some non-atomic mean measure $\mu$. 
For such a process, the number of points in disjoint measurable subsets of $\R$ are independent 
and
\begin{align}
 \mathbb{P}[\Pi(a,b)=k] = \frac{\mu(a,b)}{k!}e^{-\mu(a,b)}
\end{align}
for any $a<b$ and any $k\geq 0$, where, for any measurable $A\subseteq \R$, $\Pi(A)=\Pi A$ is the cardinality of the restriction of $\Pi$ to the set $A$.
This means that the number of points in any interval $(a,b)$ is distributed like $\text{Poi}(\mu(a,b))$. 
Sometimes, but not always, we will assume that the mean measure $\mu$ is absolutely continuous and given in terms 
of a measurable intensity function 
$\lambda : \R \to [0,\infty)$,  so that
\begin{align}
 \mu(A) = \int_A \lambda(x) \d x
\end{align}
for any measurable $A\subseteq \R$. 

Throughout we let $S_0=0$ (note that $0\notin \Pi$ with probability $1$), so that the walk starts in the origin.
The process $(S_n)_{n=0}^{\infty}$ will be referred to as {\tt GWIPP}. 
If we want to emphasise the underlying point process, the underlying mean measure, or the underlying intensity function, 
we write {\tt GWIPP}$(\Pi)$, {\tt GWIPP}$(\mu)$ or {\tt GWIPP}$(\lambda)$, respectively. 
We say that the walk jumped over $0$ if $\sign(S_{n+1})\neq \sign(S_n)$ for some $n$.
If the event $\{\sign(S_{n+1})\neq \sign(S_n) \text{ i.o.}\}$ occurs, 
then we say that {\tt GWIPP} is \emph{recurrent}. Otherwise we say that {\tt GWIPP} is \emph{transient}. 
This choice of notation is explained by viewing each jump over $0$ as a pseudo--visit at $0$.

To avoid certain degenerate cases, we will typically impose the following two conditions on the measure $\mu$.
\begin{enumerate}[(i)]
 \item $\mu(-\infty,0)=\mu(0,\infty)=\infty$. \label{meas:i}
 \item $\mu(A)<\infty$ for all bounded measurable $A\subseteq \R$. \label{meas:ii}
\end{enumerate}
Denote by $\cM$ the set of all measures on $\R$ which satisfy (\ref{meas:i}) and (\ref{meas:ii}).
Note that the first condition is equivalent to $\Pi(-\infty,0)=\Pi(0,\infty)=\infty$ with probability $1$; in Remark \ref{rem:notinf} we consider a case when this condition is not satisfied. The second condition is equivalent to $\Pi(A)<\infty$ with probability $1$, for any bounded measurable $A\subseteq \R$, which implies that there are no accumulation points of the process. Indeed, if a process has accumulation points, it is possible that the $\argmin$ in the  definition of the greedy walk is not well-defined.

If $\mu \in \cM$, then {\tt GWIPP}$(\mu)$ is recurrent if and only if $\Pi_{\infty}=\emptyset$, and {\tt GWIPP}$(\mu)$ is transient if and only if $\Pi_{\infty}\neq \emptyset$. Thus the dichotomy between recurrence and transience translates to a dichotomy between ``visits all points'' and ``does not visit all points''.  This stems from the fact that {\tt GWIPP} essentially can behave in two ways. Either the points of $\Pi$ are eventually dense enough that {\tt GWIPP} eventually gets stuck on either the positive or negative half-line and go towards $\infty$ or $-\infty$ accordingly (i.e.\ transient), or the points of $\Pi$ are sparse enough that there are infinitely many ``sufficiently long'' empty intervals on both 
half-lines, and that {\tt GWIPP} switches sign infinitely many times and thus visits all points of $\Pi$ (i.e.\ recurrent). Moreover, if {\tt GWIPP}$(\mu)$ is transient and $\mu$ symmetric around zero, then, by symmetry, {\tt GWIPP}$(\mu)$ goes to $+\infty$ or $-\infty$ with probability $1/2$ each.

The aim of this paper is to characterise (in terms of $\mu$ or $\lambda$) when {\tt GWIPP} is recurrent or transient. The following result does precisely this.

\begin{theorem}
\label{condition}
Let $\mu \in \cM$. Then {\tt GWIPP}$(\mu)$ is recurrent with probability 1 if 
$$\int_0^\infty \exp(-\mu(x,2x+R))\mu(\d x)=\infty \quad \text{and}\quad \int_{-\infty}^0 \exp(-\mu(2x-R,x))\mu(\d x)=\infty, $$
for all $R\geq 0$. If either integral is finite for some $R\geq 0$, then {\tt GWIPP}$(\mu)$ is transient with probability $1$.

Let $\lambda\in \cM$. Then {\tt GWIPP}$(\lambda)$ is recurrent with probability 1 if 
$$\int_0^\infty \!\! \exp\left(-\!\int_{x}^{2x+R}\!\!\!\! \lambda(t) \d t \right)\lambda(x) \d x=\infty \quad \text{and} \quad \int_{-\infty}^{0}\! \!\!\!\exp\left(-\!\int_{2x-R}^x\!\!\lambda(t) \d t \right)\lambda(x) \d x=\infty,$$
for all $R\geq 0$. If either integral is finite for some $R\geq 0$, then  {\tt GWIPP}$(\lambda)$ is transient with probability $1$. 
\end{theorem}
The proof of this, presented in Section \ref{sec:2}, is an application of Campbell's theorem (which allows one to determine whether random sums over Poisson processes are convergent or divergent) and the Borel--Cantelli lemmas.

We remark a few things. First, the second part of the theorem is an immediate consequence of the first. Second, this result also states that recurrence (or transience) is a zero-one event. Third, it is a straightforward consequence that taking $\lambda(t)=c$ for all $t\in \R$, i.e.\ taking $\Pi$ to be a homogenous Poisson process with rate $c$, results in {\tt GWIPP}$(\lambda)$ being transient. 
This is well-known and another proof appears in \cite{Ga}.

Take now two point processes $\Pi$ and $\Pi_0$, and assume {\tt GWIPP}$(\Pi)$ is transient. Consider $\Pi'=\Pi+\Pi_0$. 
Intuitively, adding more points to an already transient process only makes it ``more'' transient, 
since it will be more difficult to find long empty intervals which allow $(S_n)_{n=1}^{\infty}$ to change sign. 
Conversely, removing points from an already recurrent process makes it ``more'' recurrent.
Since recurrence (or transience) does not depend on the point process in any finite interval around $0$, as the following result shows,  it suffices to look at what happens far away from the origin. (Equivalently, the convergence or divergence of the integrals in Theorem \ref{condition} depends only on the tail behaviour.) The proof appears in Section \ref{sec:2}.

\begin{lemma}\label{lem:coupling}
Let $\mu,\mu' \in \cM$ and suppose there is some 
$K>0$ such that $\mu'(A)\geq \mu(A)$ for all measurable $A\subseteq (-\infty,-K)\cup (K,\infty)$. 
If {\tt GWIPP}$(\mu')$ is recurrent with probability $1$, then {\tt GWIPP}$(\mu)$ is recurrent with probability $1$.
\end{lemma}

Theorem \ref{condition} also facilitates the identification of threshold functions for recurrence, and it transpires that the iterated logarithms are useful in this context. Before we state our next result, we need some definitions. We define the ``power tower'' recursively by $a \uparrow \uparrow 0 := 1$ and  $a\uparrow \uparrow n := a^{a\uparrow\uparrow (n-1)}$ for any $a\in [0,\infty)$ and  $n\geq 1$. Let $\log$ be the ordinary natural logarithm. We define the iterated logarithm $\log^{(n)}$, for $n\geq 1$, to be the function defined recursively by
\begin{align}
 \log^{(1)}t := \begin{cases} \log t  & \text{ if } t> 1 \\ 0 & \text{ otherwise,} \end{cases}
\end{align}
and, for any $n\geq 2$, 
\begin{align}
\log^{(n)} t := \log^{(1)} \left(\log^{(n-1)} t \right).
\end{align}
Note that $\log^{(n)}t = 0$ for any $t\leq e \uparrow \uparrow (n-1)$.

\begin{proposition}\label{prop:thresholdregion}
 Let 
\begin{align}
\lambda(t) := \frac{1}{|t|\log 2}\sum_{i=2}^n a_i \log^{(i)}|t|.
\end{align}
where $n\in \{2,3,4,\dots\}$ and $a_i\geq 0$ for all $2\leq i \leq n$. Then {\tt GWIPP}$(\lambda)$ is transient with probability $1$ if 
\begin{itemize}
 \item $a_2 > 1$, or 
 \item $a_2=1,a_3>2 $, or 
 \item $a_2=1,a_3=2$, and there exists some $m\geq 4$ such that $a_4=1,a_5=1, \dots, a_m=1$ and $a_{m+1}>1$.
\end{itemize}
Otherwise, {\tt GWIPP}$(\lambda)$ is recurrent with probability $1$.
\end{proposition}

One could also ask to what extent Proposition \ref{prop:thresholdregion} extends to the infinite case. The statement about transience remains true, but the final statement about recurrence does not, unless one introduces some (rather mild) restrictions on the growth rate of the sequence $(a_n)_{n=2}^{\infty}$. Consider
\begin{align}
  \lambda(t) = \frac{1}{|t|\log 2}\sum_{i=2}^{\infty} a_i \log^{(i)}|t|. \label{eq:infsum}
\end{align}
If $a_2=0$ and $a_n=(e\uparrow \uparrow n)$ for $n\geq 3$, then {\tt GWIPP}$(\lambda)$ is transient with probability $1$, even though $a_2=0$. However, if $a_2=0$ and $a_n=(2\uparrow \uparrow n)$ for $n\geq 3$, then {\tt GWIPP}$(\lambda)$ is recurrent with probability $1$.

The next result describes the threshold between transience and recurrence in even greater detail. In particular, it shows that taking $a_3=2$ and $a_2=1=a_4=a_5=\dots $ in \eqref{eq:infsum} means that {\tt GWIPP}$(\lambda)$ is recurrent with probability $1$.

\begin{proposition}\label{prop:slowerthanlog2}
 Let $a_3=2$ and $a_2=1=a_4=a_5=\dots $ and let $(b_n)_{n=1}^{\infty}$ be non-decreasing sequence satisfying $b_n=O(e\uparrow \uparrow (n-2))$, $g:(0,\infty)\to [1,\infty)$ be a non-decreasing slowly varying function satisfying $g(e \uparrow \uparrow n)=b_n$, and let
\begin{align}
\lambda(t) := \frac{1}{|t|\log 2}\left(\sum_{i=2}^{\infty} a_i \log^{(i)}|t|+\log^{(1)} g(|t|)\right).
\end{align}
If $\sum_{n=2}^{\infty}1/b_n=\infty$, then {\tt GWIPP}$(\lambda)$ is recurrent with probability $1$. If $\sum_{n=2}^{\infty}1/b_n<\infty$, then {\tt GWIPP}$(\lambda)$ is transient with probability $1$.
\end{proposition}

The following result provides a useful tool for investigating the behaviour of a given intensity function. The idea behind the proof is essentially to find a suitable intensity function for comparison, and apply Lemma \ref{lem:coupling} and Proposition \ref{prop:thresholdregion}.

\begin{proposition}\label{thm:threshold}
 Let $\lambda\in \cM$. Let $a_3=2$ and $a_2=1=a_4=a_5=\dots $. If there exists some $n\geq 2$ such that
\begin{align}
 \liminf_{t\to \infty} \frac{t\lambda(t)\log 2 - \sum_{i=2}^{n-1}a_i\log^{(i)}t}{a_n\log^{(n)}t}>1 \quad \text{or} \quad  \liminf_{t\to -\infty} \frac{|t|\lambda(t)\log 2 - \sum_{i=2}^{n-1}a_i\log^{(i)}|t|}{a_n\log^{(n)}|t|}>1
\end{align}
then {\tt GWIPP}$(\lambda)$ is transient with probability $1$.
If there exists some $n\geq 2$ such that
\begin{align}
 \limsup_{t\to \infty} \frac{t\lambda(t)\log 2 - \sum_{i=2}^{n-1}a_i\log^{(i)}t}{a_n\log^{(n)}t}<1 \quad \text{and} \quad  \limsup_{t\to -\infty} \frac{|t|\lambda(t)\log 2 - \sum_{i=2}^{n-1}a_i\log^{(i)}|t|}{a_n\log^{(n)}|t|}<1
\end{align}
then {\tt GWIPP}$(\lambda)$ is recurrent with probability $1$.
\end{proposition}

Proposition \ref{thm:threshold} does not answer what happens if, say,
\begin{align}
 \lim_{t\to \infty} \frac{t\lambda(t)\log 2 - \sum_{i=2}^{n-1}a_i\log^{(i)}t}{a_n\log^{(n)}t}=1 \quad \text{and} \quad  \lim_{t\to -\infty} \frac{|t|\lambda(t)\log 2 - \sum_{i=2}^{n-1}a_i\log^{(i)}|t|}{a_n\log^{(n)}|t|}=1
\end{align}
for all $n\geq 2$. As seen in Proposition \ref{prop:slowerthanlog2}, both recurrence and transience are possible in this case.

The remainder of this paper is outlined as follow. In Section \ref{sec:2} we prove mainly general results, including Theorem \ref{condition} and Lemma \ref{lem:coupling}. In Section \ref{sec:3} we concentrate on more concrete threshold results, i.e.\ Propositions \ref{prop:thresholdregion}--\ref{thm:threshold} along with related results.

\section{Proof of Theorem \ref{condition} and related results}\label{sec:2}
Throughout we will write $\Pi=\{X_i  :  i\in \Z\setminus \{0 \}\}$, assuming as we may that
\begin{align}
\cdots < X_{-2}<X_{-1}<0<X_1<X_2< \cdots.
\end{align}
For $k>0$, let
$$A_k^R=\{\Pi(X_k,2X_k+R)=0\}=\{d(X_k,-R)<d(X_k,X_{k+1})\}$$
and
$$B_k^R=\{\Pi(2X_{-k}-R,X_{-k})=0\}=\{d(X_{-k},R)<d(X_{-k},X_{-k-1})\}.$$
 
The following lemma describes the connections between these events and recurrence of {\tt GWIPP}.

\begin{lemma}\label{lem:reciff}

With probability 1,
\begin{align*}
\{{\tt GWIPP} \text{ recurrent}\}=\bigcap_{R\geq 0}\{A^R_k \text{ i.o.}, B^R_k \text{ i.o.} \}.
\end{align*}
\end{lemma}
\proof
With probability 1, $\Pi(A)<\infty$ for any finite set $A$ and for all $n$ there is a unique point which is the closest unvisited point to $S_n$. 
On this event, the walk is well-defined and $\vert S_n \vert \rightarrow \infty$. 

Then, either
$\{\vert S_n\vert\rightarrow\infty \text{ but } S_nS_{n+1}<0 \text{ i.o.}\}$ occurs (i.e.\ {\tt GWIPP} is recurrent), or
$\{S_n\rightarrow \infty\}\cup\{S_n\rightarrow -\infty\}$ occurs (i.e.\ {\tt GWIPP} is transient).

Suppose first that $\{A^R_k \text{ i.o.}\}$ and $\{B^R_k \text{ i.o.}\}$ occur for all $R\geq 0$, and for contradiction that {\tt GWIPP} is transient.
Without loss of generality, we may assume $S_n\to \infty$ as $n\to \infty$. 
Then there exists $J$ such that $S_{n+1}>S_n$ for all $n>J$. 
Let $Y=\max \{X\in \Pi  :  X<  \min_{0\leq k\leq J}S_k \}$ be the rightmost point never visited (such a point exists because of the assumptions of transience and $S_n\to \infty$ as $n\to \infty$)
and choose $R$ such that $R>|Y|$. 
Note that $Y$ is the closest unvisited point to the left of $S_n$, when $n\geq J$.
Since, by assumption, $A^R_k$ occurs for some $k>J$, we have $d(X_k,Y)<d(X_k,-R)<d(X_k,X_{k+1})$.
Moreover, there exists $n$ such that $S_{n}=X_k$ and by the definition of the greedy walk $S_{n+1}=Y$, which is a contradiction. Hence {\tt GWIPP} recurrent.

For the other direction, 
assume that {\tt GWIPP} is recurrent, but, for contradiction, that $A_k^R$ (the argument being identical for $B_k^{R}$) occurs at most finitely many times for some $R\geq 0$.
Let $J=\max\{k\in\Z: X_k<-R\}$ and let $K=\max\{k>0: A_k^R \text{ occurs}\}$. 
Then for all $k>K$, $d(X_k,X_{k+1})<d(X_k,-R)$.
Since {\tt GWIPP} is recurrent, it visits all points of $\Pi$. In particular, there is a finite time $N$ after which all points in $(X_J,X_k]\subset [-R,X_K]$ have been visited.
But then, for all $n>N$ such that $S_n>0$ and $S_{n+1}<0$, we have $S_n=X_k$ for some $k>K$ and $d(X_k,-R)<d(X_k,X_{J})\leq d(X_k,S_{n+1})<d(X_k,X_{k+1})$. This contradicts $d(X_k,X_{k+1})<d(X_k,-R)$.

\qed

This characterisation suggests that the Borel--Cantelli lemmas will be useful. 
In particular, we use the extended Borel--Cantelli Lemma.

\begin{lemma}[Extended Borel--Cantelli lemma, {\cite[Corollary 6.20]{Ka}}]
\label{extended}
Let $\mathcal{F}_n$, $n\geq 0$, be a filtration and let $A_n\in\mathcal{F}_n$, $n\geq 1$.  Then, with probability 1,
\begin{align}
 \{A_n \text{ i.o.}\}=\left\{\sum_{n=1}^{\infty}\mathbb{P}[A_n~\vert~\mathcal{F}_{n-1}]=\infty\right\}.
\end{align}

\end{lemma}

The convergence or divergence of the associated random series will be determined using Campbell's theorem for sums of non-negative measurable functions, which provides a zero-one law for the convergence of a random series.

\begin{theorem}[Campbell's theorem, {\cite[Section 3.2]{Kingman}}]
\label{Campbell}
 Let $\Pi$ be a Poisson process on $S$ with mean measure $\mu$ and let
$f:S\rightarrow [0,\infty]$ be a measurable function. Then the sum
$$\sum_{X\in\Pi}f(X)$$
is convergent with probability 1 if and only if
\begin{align}
 \label{convergence}
\int_S \min\{ f(x),1\}\mu(\d x)<\infty.
\end{align}
Moreover, the sum diverges with probability 1 if and only if the integral diverges.
\end{theorem}

We are now ready to prove Theorem \ref{condition}. 

\proof [Proof of Theorem \ref{condition}]
From Lemma \ref{lem:reciff} it follows that the sufficient and necessary conditions implying
that the {\tt GWIPP} is recurrent with probability 1,
are the same as those implying that $\{A_k^R \text{ i.o.}\}$ and $\{B_k^R \text{ i.o.}\}$ occur with probability 1 for all $R\geq 0$.

Let $\mathcal{F}_k=\sigma(X_1,X_2,\dots,X_k)$.
Then $A_k^{R}\in \mathcal{F}_{k+1}$ for any $R\geq 0$, and
\begin{align*}
 \mathbb{P}[A_k^R~\vert~\mathcal{F}_k]= \mathbb{P}[\Pi(X_k,2X_k+R)=0~\vert~\mathcal{F}_k]
=\exp (-\mu(X_k,2X_k+R)),
\end{align*}
where the final equality holds since $X_k\in \mathcal{F}_k$, $\Pi\cap{(X_k,\infty)}$ is independent of $\mathcal{F}_k$ 
and the number of points in a measurable set $A\subseteq \R$ is distributed like $\text{Poi}(\mu(A))$.
Applying  Theorem \ref{Campbell} with $f(x)=\exp(-\mu(x,2x+R))$, we obtain
$$\sum_{k=1}^{\infty}\mathbb{P}[A_k^R~\vert~\mathcal{F}_k]=\sum_{k=1}^{\infty}\exp (-\mu(X_k,2X_k+R))=\infty $$
with probability $1$ if and only if
\begin{align}
 \int_0^\infty \exp(-\mu(x,2x+R))\mu(\d x)=\infty.
\end{align}
Moreover,  Lemma \ref{extended} implies that $\sum_{k=1}^{\infty}\mathbb{P}[A_k^R~\vert~\mathcal{F}_k]=\infty$ a.s.\
 if and only if  $\mathbb{P}[A_k^R \text{ i.o.}]=1$.
Thus, the integral above diverges if and only if $\mathbb{P}[A_k^R \text{ i.o.}]=1$.

Similarly, if the integral above converges, so does the sum 
$$\sum_{k=1}^{\infty}\mathbb{P}[A_k^R~\vert~\mathcal{F}_k]$$
with probability $1$, and then by Lemma \ref{extended}, the event $\{A_k^R \text{ i.o.}\}$ does not occur with probability $1$.

In the same way one can show that 
$\int_{-\infty}^0 \exp(-\mu(2x-R,x))\mu(\d x)=\infty$
if and only if $\mathbb{P}[B_k^R \text{ i.o.}]=1$; and, conversely, 
if the integral converges, then $\mathbb{P}[B_k^R \text{ i.o.}]=0$.

\qed

\begin{remark}
We lose no generality by assuming that the greedy walk on $\Pi$ starts from the origin, since recurrence/transience does not depend on the starting point.
One explanation of this is that the distribution of the points in any finite interval around the origin does not influence the behaviour of the greedy walk far away from the origin.
More precisely, suppose the walk starts from $a\in \R$, $a>0$ (one can argue similarly for $a<0$). Then, for any $R\geq 0$, one can show that the events 
$\{\Pi(X_k,2X_k-a+R)=0\}$ and $\{\Pi(2X_{-k}-a-R,X_{-k})=0\}$ occur for infinitely many $k$ if and only if $\{A_k^R \text{ i.o.}\}$ and $\{B_k^R \text{ i.o.}\}$ occur.
As we have seen in Lemma \ref{lem:reciff}, {\tt GWIPP}$(\Pi)$ is recurrent if these events occur.
\end{remark}

In the following remark we explore what happens if $\mu$ does not satisfy condition (\ref{meas:i}).

\begin{remark}\label{rem:notinf}
If $\mu(-\infty,0)<\infty$ and $\mu(0,\infty)=\infty$, then it is not true that {\tt GWIPP}$(\mu)$ is transient
 if and only if $\Pi_{\infty}\neq \emptyset$. 
 (This should be contrasted with the situation when $\mu \in \cM$.) 
 To see this, consider the intensity function $\lambda(t)=\mathbbm{1}_{(-M,\infty)}(t)$, where $M>0$ will be chosen later.
 It is clear that ${\tt GWIPP}(\lambda)$ is transient with probability $1$, for any $M$, but we will show that $M$ can be chosen so that $0<\mathbb{P}[\Pi_{\infty}\neq \emptyset]<1$. Note that 
\begin{align}
 \mathbb{P}[\Pi_{\infty}=\Pi_{\infty}\cap{(-M,0)}]=1. 
\end{align}
so we may consider $\Pi_{\infty}\cap{(-M,0)}$ instead of $\Pi_{\infty}$.

 Let $\lambda'(t)=1$ for all $-\infty<t<\infty$. Denote by $\Pi'$ the associated Poisson process and by $(S_n')_{n=1}^{\infty}$ the associated greedy walk. We can couple $\Pi$ and $\Pi'$ so that $\Pi=\Pi'\cap{(-M,\infty)}$. 
 By symmetry, we have
\begin{align}
 \mathbb{P}[S_n'\to \infty]=\frac{1}{2} =  \mathbb{P}[S_n'\to -\infty].
\end{align}
It holds that
\begin{align}
 \mathbb{P}[\Pi'_{\infty}\cap{(-M,0)}=\emptyset\  |\  S_n'\to -\infty] = 1,
\end{align}
and
\begin{align}
  \mathbb{P}[\Pi'_{\infty}\cap{(-M,0)}=\emptyset \  |  \ S_n'\to \infty] > 0.
\end{align}
By the coupling of $\Pi$ and $\Pi'$, it holds that $\Pi'_{\infty}\cap{(-M,0)}= \Pi_{\infty}\cap{(-M,0)}$ almost surely, whence
\begin{align}
  \mathbb{P}[\Pi_{\infty}\cap{(-M,0)}=\emptyset] & = \mathbb{P}[ \Pi'_{\infty}\cap{(-M,0)}=\emptyset] \\
& =  \mathbb{P}[\Pi'_{\infty}\cap{(-M,0)}=\emptyset\  | \ S_n'\to -\infty]\mathbb{P}[S_n'\to -\infty] \\ & \qquad + \mathbb{P}[\Pi'_{\infty}\cap{(-M,0)}=\emptyset\  | \ S_n'\to \infty]\mathbb{P}[S_n'\to \infty]\\ &> \frac{1}{2}.
\end{align}
Therefore $\mathbb{P}[\Pi_{\infty}\cap{(-M,0)}\neq \emptyset]<\frac{1}{2}$.

For the other inequality, we first have
\begin{align}
 \mathbb{P}[\Pi_{\infty}\cap{(-M,0)} = \Pi\cap{(-M,0)}] 
&= \mathbb{P}[\Pi'_{\infty}\cap{(-M,0)} = \Pi'\cap{(-M,0)}] \\ 
& \geq  \mathbb{P}[\Pi'_{\infty}\cap{(-\infty,0)} =\Pi'\cap{(-\infty,0)}] \\ & = \prod_{n=1}^{\infty}\left(1-\frac{1}{2^n} \right) \approx 0.288\dots ,
\end{align}
where the final equality follows from \cite[Theorem 1]{Ga}. Now pick $M$ large enough that 
\begin{align}
 \mathbb{P}[\Pi\cap{(-M,0)}\neq \emptyset]>1-\prod_{n=1}^{\infty}\left(1-\frac{1}{2^n} \right).
\end{align}
Then
\begin{align}
 \mathbb{P}[\Pi_{\infty}\cap{(-M,0)} \neq \emptyset] 
 &\geq \mathbb{P}[\Pi_{\infty}\cap{(-M,0)} = \Pi\cap{(-M,0)},\ \Pi'_{\infty}\cap{(-M,0)} \neq \emptyset ] \\
 &\geq \mathbb{P}[\Pi_{\infty}\cap{(-M,0)} = \Pi\cap{(-M,0)}]+\mathbb{P}[\Pi'_{\infty}\cap{(-M,0)} \neq \emptyset ]-1\\
 & >0.
\end{align}

This implies that $0< \mathbb{P}[\Pi_{\infty} = \emptyset] <\frac{1}{2}$, even though {\tt GWIPP}$(\lambda)$ is transient with probability 1.
\end{remark}

A natural question is which conditions one needs to place on $\mu$ (or $\lambda$) so that $\{A_k^R \text{ i.o.}\}$ for all $R>0$ if and only if $\{A_k^0 \text{ i.o.}\}$. The reason why this is not an unreasonable demand is that the events $\Pi(X_k,2X_k+R)=0$ and $\Pi(X_k,2X_k)=0$ should not be too different for large $X_k$, since the length of the interval $(2X_k,2X_k+R)$ becomes negligible compared to the length of $(X_k,2X_k)$ in the limit. However, the following example shows that some extra conditions need to be placed, and that, in general, $\{A_k^0 \text{ i.o.}\}$ does not imply that $\{A_k^R \text{ i.o.}\}$ for all $R>0$. 

\begin{remark} \label{peaks}
Let
\begin{align}
\lambda(t)= \sum_{n=1}^{\infty}a_n \mathbbm{1}(2^n-2<t<2^n-1)
\end{align}
for some increasing sequence $(a_n)_{n=1}^{\infty}$. For $n=1,2,\dots$, denote by $C_n$ the event that $\Pi(2^n-2,2^n-1)=0$.

If $X$ equals the rightmost point in the interval $(2^n-2,2^n-1)$, 
then $\Pi(X,2X)=0$ almost surely.
This implies that $X$ is always closer to $0$ than to the leftmost point in $(2^{n+1}-2,2^{n+1}-1)$. 
Hence, $\{A_n^0 \text{ i.o.}\}$ occurs with probability 1. 

However, for $R=3$ we have $\{ A_n^3 \text{ i.o.}\} \subseteq \{C_n \text{ i.o.}\}$. 
Choose now the sequence $(a_n)_{n=1}^{\infty}$ such that
\begin{align}
 \sum_{n=1}^{\infty}\mathbb{P}[C_n] = \sum_{n=1}^{\infty}e^{-a_n} < \infty.
\end{align}
By the Borel--Cantelli lemma, the probability of $\{C_n \text{ i.o}\}$ is $0$, 
which implies that also $\P( A_n^3 \text{ i.o.})=0$. Therefore {\tt GWIPP}($\lambda)$ is transient even though $A_n^0$ occurs infinitely often with probability $1$.
\end{remark}

Denote by $\cM_b\subseteq \cM$ those measures $\mu\in \cM$ with the property that for any $R\geq 0$, 
there exists some constant $C=C(R)>0$, such that $\mu(x,x+R)<C$ and $\mu(-x-R,-x)<C$ for all $x\geq 0$.
\label{meas:iii} 
As the following lemma shows, this boundedness assumption disallows any examples of the type in Remark \ref{peaks}.

\begin{lemma}
 Let $\mu\in \cM_b$. Then {\tt GWIPP}$(\mu)$ is recurrent with probability $1$ if
\begin{align}
 \int_0^\infty \exp(-\mu(x,2x))\mu(\d x)=\infty \quad \text{and}\quad \int_{-\infty}^0 \exp(-\mu(2x,x))\mu(\d x)=\infty.
\end{align}
If either integral is finite, then {\tt GWIPP}$(\mu)$ is transient with probability $1$.
\end{lemma}

\proof
Fix $R>0$. We have
\begin{align*}
 \exp(-C)\int_0^\infty \exp(-\mu(x,2x))\mu(\d x)
&\leq \int_0^\infty \exp(-\mu(x,2x)-\mu(2x,2x+R))\mu(\d x) \\
& = \int_0^\infty \exp(-\mu(x,2x+R))\mu(\d x) \\ 
& \leq \int_0^\infty \exp(-\mu(x,2x))\mu(\d x).
\end{align*}
The integral on the negative half-line can be similarly bounded. 
Therefore the integrals in the statement of the lemma diverge if and only if the corresponding integrals in Theorem \ref{condition} diverge. 
This proves the claim.
\qed

For instance, if $\mu\in \cM$ and the maps $x\mapsto \mu(0,x)$ and $x \mapsto \mu(-x,0)$ from $[0,\infty)$ to $[0,\infty)$ 
are Lipschitz, then $\mu\in \cM_b$.
Also, $\limsup_{t\rightarrow \pm \infty}\lambda(t)<\infty$ implies that $\lambda\in \cM_b$, which gives the following corollary. 

\begin{corollary}\label{cor:intcond}
 Suppose $\lambda\in \cM$ and $\limsup_{t\rightarrow \pm \infty}\lambda(t)<\infty$. Then {\tt GWIPP}$(\lambda)$ is recurrent with probability $1$ if 
\begin{align}
 \int_0^{\infty}\!\! \exp\left( - \int_x^{2x} \lambda(t)\d t \right) \lambda(x) \d x = \infty
\quad \text { and } \quad  \int_{-\infty}^0 \!\! \exp\left( - \int_{2x}^{x} \lambda(t)\d t \right) \lambda(x) \d x = \infty.
\end{align}
If either integral is finite, then {\tt GWIPP}$(\lambda)$ is transient with probability $1$.
\end{corollary}

Next we prove Lemma \ref{lem:coupling}.
\begin{proof}[Proof of Lemma \ref{lem:coupling}]
Denote by $\Pi$ the point process with mean measure $\mu$ and let $(S_n)_{n=0}^{\infty}$ be {\tt GWIPP}$(\Pi)$. 
Similarly, denote by $\Pi'$ the point process with mean measure $\mu'$ and let $(S'_n)_{n=0}^{\infty}$ be {\tt GWIPP}$(\Pi')$. 
Denote the points of $\Pi$ and $\Pi'$ 
by 
\begin{align}
\cdots < X_{-2}<X_{-1}<0<X_1<X_2< \cdots  \text{ and } \cdots < X'_{-2}<X'_{-1}<0<X'_1<X'_2< \cdots
\end{align}
respectively.
 Since $\mu'(A)\geq \mu(A)$ for all measurable $A\subset (-\infty,-K)\cup (K,\infty)$,
 we can couple $\Pi$ and $\Pi'$ together so that 
$x\in ((-\infty,-K)\cup (K,\infty)) \cap \Pi$ implies that $x\in \Pi'$.

Assume, for contradiction, that {\tt GWIPP}$(\Pi')$ is recurrent and {\tt GWIPP}$(\Pi)$ is transient. 
Without loss of generality, we may assume that $S_n\to \infty$ as $n\to \infty$.
Then there is some $M_0\geq 1$ 
such that $S_{k+1}>S_k$ for all $k>M_0$, i.e.\ $(S_n)_{n=1}^{\infty}$ moves only to the right after time $M_0$. 
Assume moreover that $M_0$ is large enough that $S_{M_0}>K$, so that we are on the region where $\Pi$ and $\Pi'$ are coupled.

For the remainder of the proof, see Figure \ref{fig:coupling} for an illustration.
Let $Y=\max\{X\in\Pi:X<\min_{0\leq k\leq M_0}{S_k}\}$, that is, let $Y$ be the rightmost point of $\Pi$ that is never visited.
Note that $Y$ is well-defined 
because of the transience of {\tt GWIPP}$(\Pi)$ and the assumption $S_n\to \infty$ as $n\to \infty$.

Since {\tt GWIPP}$(\Pi')$ is recurrent, $(S_n')_{n=1}^{\infty}$ visits all points of $(K,\infty)\cap \Pi\subseteq (K,\infty)\cap \Pi'$ and jumps over $0$ infinitely often.
Thus we can find $J\geq 1$ such that $S_J'>S_{M_0}$ and $S'_{J+1}<Y$.
Let $S_J'=X_k'$ and let $\ell$ be such that $X_{\ell}\leq X_k'< X_{\ell+1}$. 
Moreover, let $M>M_0$ be such that
$S_M=X_{\ell}$ and $S_{M+1}=X_{\ell+1}$. 

\begin{figure}[ht]
 \includegraphics[width=0.9\textwidth]{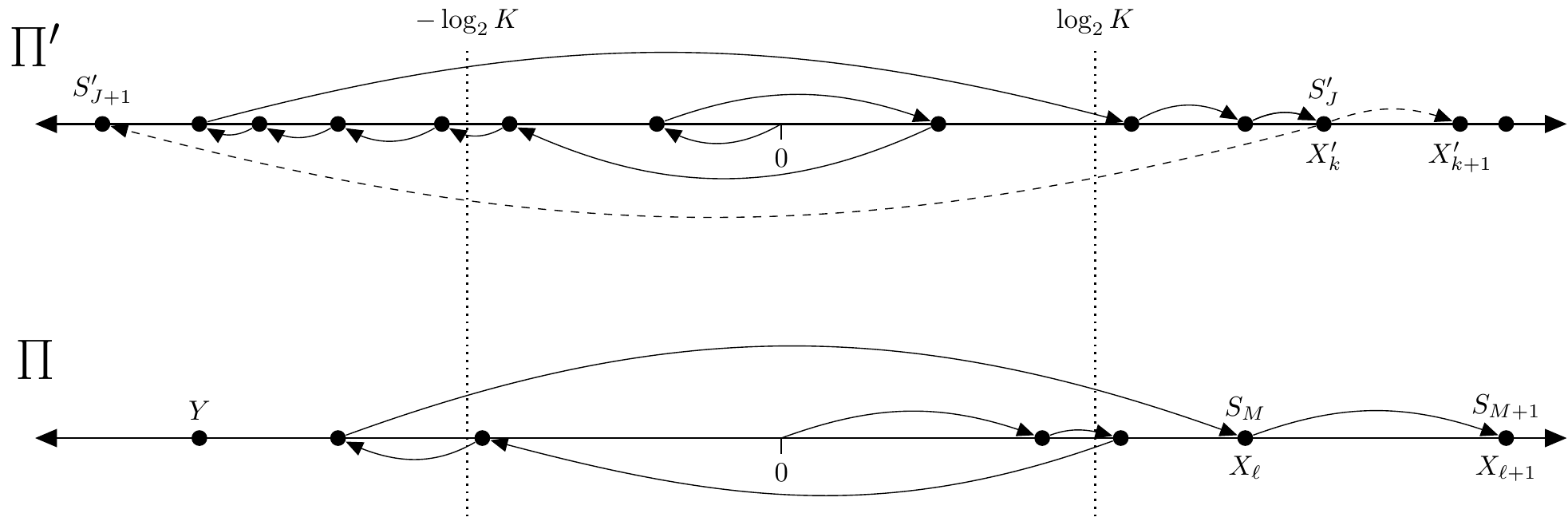}
\caption{An illustration of the proof of Lemma \ref{lem:coupling}. 
Note that both the positive and negative axis have been rescaled logarithmically. 
The proof shows that $S_{J+1}'=X_{k+1}'$ is forced, which contradicts the choice of $J$ 
(which implies that $S_{J+1}'<0$).}
\label{fig:coupling}
\end{figure}

The coupling between $\Pi$ and $\Pi'$ on $(K,\infty)$ implies that $X_{\ell}\leq S_J'<X'_{k+1}\leq X_{\ell+1}$.
Therefore $d(S_J',X'_{k+1})\leq d(S_{M},S_{M+1})<d(S_M,Y)\leq d(S'_{J},S_{J+1}')$,
which contradicts the choice of $J$, that is $S'_{J+1}<Y$. 
Thus, if {\tt GWIPP}($\Pi')$ is recurrent, so is {\tt GWIPP}$(\Pi)$.
\end{proof}

A question that complements Lemma \ref{lem:coupling} is the following. Suppose {\tt GWIPP}$(\lambda')$ is recurrent and let $\lambda=\lambda'+\lambda_0$ for some intensity function $\lambda_0$. Which conditions on $\lambda_0$ should one place to ensure that {\tt GWIPP}$(\lambda)$ is also recurrent? That is, how many points, and where, can we add to a recurrent process without making it transient? 
The following lemma yields a partial answer to this.

\begin{lemma}\label{lem:addpoints}
Suppose {\tt GWIPP}$(\lambda)$ is recurrent with probability $1$. If, for all $R\geq 0$,
\begin{align}
 \limsup_{x\to \infty}\int_{x}^{2x+R}\max \left(0, \lambda'(t)-\lambda(t)\right)\d t <\infty,
\end{align}
and 
\begin{align}
 \limsup_{x\to -\infty}\int_{2x-R}^{x}\max \left(0, \lambda'(t)-\lambda(t)\right)\d t <\infty,
\end{align}
then {\tt GWIPP}$(\lambda')$ is recurrent with probability $1$.  
\end{lemma}

\begin{proof}
For all $t\in \R$, let $\lambda''(t) = \max(\lambda(t),\lambda'(t))$. It suffices to show that {\tt GWIPP}($\lambda''$) is recurrent, since then, by Lemma \ref{lem:coupling}, {\tt GWIPP}$(\lambda')$ also is recurrent.
The first condition implies that for any $R\geq 0$ there exists some $C>0$ such that
\begin{align}
\int_{x}^{2x+R}( \lambda''(t)-\lambda(t))\d t < C
\end{align}
for all all large enough $x>0$. Therefore
\begin{align}
 \int_{x}^{2x+R}\lambda''(t)\d t = \int_x^{2x+R}\lambda(t) \d t + \int_{x}^{2x+R}(\lambda''(t)-\lambda(t))\d t < \int_x^{2x+R}\lambda(t) \d t + C
\end{align}
for large enough $x$. For $M$ large enough,
\begin{align}
 \int_{M}^{\infty} \exp\left(-\int_x^{2x+R}\lambda''(t) \d t \right) \lambda''(x) \d x 
& \geq \int_{M}^{\infty} \exp\left(-\int_x^{2x+R}\lambda''(t) \d t \right) \lambda(x) \d x \\
& \geq \int_M^{\infty}\exp\left(-\int_x^{2x+R}\lambda(t) \d t -C \right) \lambda(x) \d x \\
& = \exp(-C) \int_M^{\infty}\exp\left(-\int_x^{2x+R}\lambda(t) \d t  \right) \lambda(x) \d x \\
&=\infty,
\end{align}
where the final equality follow from the fact that {\tt GWIPP}$(\lambda)$ is recurrent
and Theorem \ref{condition}. The integral on the negative half-line can be handled similarly. 
Therefore, by Theorem \ref{condition},  {\tt GWIPP}($\lambda''$) is recurrent.
\end{proof}

\section{Treshold results}\label{sec:3}
In this section we study the threshold between transience and recurrence, proving Propositions \ref{prop:thresholdregion}--\ref{thm:threshold} and related results. We focus on symmetric intensity functions of the form 
\begin{equation}
 \label{lambda}
 \lambda_f(t)=\frac{\log f(\vert t\vert)}{\vert t\vert\log 2},
\end{equation}
where $f:(0,\infty)\rightarrow [1,\infty)$ is a regularly varying function with non-negative index $\beta$, meaning that $\lim_{t\to \infty}f(at)/f(t)=a^{\beta}$ for any $a\geq 0$. If $\beta=0$, then $f$ is said to be slowly varying.
For a thorough introduction to the theory of regular variation, we refer the reader to \cite{book:regularvariation}.

Let $\cM_s$ be the set of all intensity functions $\lambda_f\in \cM$ such that $f$ is a regularly varying function with index $\beta\geq 0$.
One can show that $\cM_s\subset \cM_b\subset \cM$, so  we may apply all results developed in Section \ref{sec:2}.
The intensity functions in $\cM_s$ are symmetric about $0$, so it suffices to only look at the positive half-line. 
However, the results in this section can easily be adapted to the case when $\lambda$ is not assumed to be symmetric.

We use the following standard notation. If $f,g:\R \to \R$ are two functions and there exists $C> 0$ such that $|f(x)|\leq C|g(x)|$ for all large enough $x$, then we write $f(x)=O(g(x))$. If $f(x)=O(g(x))$ and $g(x)=O(f(x))$, then we write $f(x)=\Theta(g(x))$.

 \begin{lemma}[{\cite[Theorem 1.5.2]{book:regularvariation}}]\label{lem:uniconv}
 If $A\subseteq (0,\infty)$ is a compact set and $f:(0,\infty)\to [0,\infty)$ is regularly varying with index $\beta$, then $f(ax)/f(x)\to a^{\beta}$ as $x\to \infty$, uniformly for all $a\in A$.
\end{lemma}

\begin{lemma}\label{lem:regvar}
   Suppose $\lambda_f\in \cM_s$. Then {\tt GWIPP}$(\lambda_f)$ is recurrent with probability $1$ if  
 \begin{align}
  \int_0^{\infty}\frac{\log f(x)}{xf(x)}\d x=\infty.
 \end{align}
 If the integral is finite, then {\tt GWIPP}$(\lambda_f)$ is transient with probability $1$.
 \end{lemma}

\begin{proof}
The set $[1,2]$ is compact and $f$ is regularly varying. It follows by Lemma \ref{lem:uniconv}, that
\begin{align}
 \int_x^{2x}\frac{\log f(t)}{t \log 2}\d t = \log f(x) + O(1).
\end{align}
Hence
\begin{align}
\int_0^{\infty} \exp\left(-\int_{x}^{2x}\lambda_f(t)\d t \right) \lambda_f(x) \d x  = \int_0^{\infty} \frac{\Theta(1)}{f(x)}\frac{\log f(x)}{x\log 2} \d x
 = \Theta(1)\int_0^{\infty}\frac{\log f(x)}{xf(x)}\d x.
\end{align}
The claim follows from Corollary \ref{cor:intcond}.
\end{proof}

 The next corollary states that if $f$ is regularly varying with positive index, then we obtain a transitive processes with probability $1$. 

\begin{corollary}
\label{reg_varying}

Let $f$ be a regularly varying function with index $\beta>0$. Then {\tt GWIPP}$(\lambda_f)$ is transient with probability $1$. 
\end{corollary}

\begin{proof}
There exists a slowly varying function $\ell(x)$ such that $f(x)=x^{\beta}\ell(x)$  (see, e.g.\ \cite[Theorem 1.4.1]{book:regularvariation}).
Then for $x>0$,
\begin{align}
 \frac{\log f(x)}{xf(x)} = \frac{\log f(x)}{x^{1+\beta}\ell(x)}=\frac{L(x)}{x^{1+\beta}}, \label{eq:slow_var}
\end{align}
where $L(x)=\frac{\log f(x)}{\ell(x)}$ is a slowly varying function (see, e.g. \cite[Theorem 1.3.6]{book:regularvariation}). 
The function on the right hand side of \eqref{eq:slow_var} is integrable on $(0,\infty)$ whenever $\beta>0$, and by Lemma \ref{lem:regvar}, {\tt GWIPP}$(\lambda_f)$ is transient. 
\end{proof}

The intensity functions in Proposition \ref{prop:thresholdregion} lie in $\cM_s$ with $f$ slowly varying, showing that the transition between recurrence and transience occurs inside the subclass of $\cM_s$ for which $f$ is slowly varying.

\begin{proof}[Proof of Proposition \ref{prop:thresholdregion}]
Let  
\begin{align}
 f(t) =\prod_{i=2}^n (\log^{(i-1)} t )^{a_i} ,
\end{align}
so that $\lambda(t)=\frac{\log f(t)}{t\log 2}$. Note that $\lambda\in \cM_s$. Assume first that $a_2>0$. Then
\begin{align}
\int \frac{\log f(x)}{xf(x)} \d x = \int \frac{\sum_{i=2}^n a_i \log^{(i)} x }{x \prod_{i=2}^n (\log^{(i-1)} x)^{a_i}} \d x = \Theta \left( \int \frac{\log^{(2)} x}{x \prod_{i=2}^{n}(\log^{(i-1)} x)^{a_i}}\d x \right),
\end{align}
where the final integral is the leading order term of the sum. The final integral is convergent precisely when one of the conditions in the statement is satisfied. (This is seen by repeatedly using the change of variables $x\mapsto e^x$.) By Lemma \ref{lem:regvar}, the statement follows.
If $a_2=0$, then consider instead $\lambda'(t):=\lambda(t)+\frac{1}{2}\frac{\log^{(2)} |t|}{|t|\log 2}$ and use the above along with Lemma \ref{lem:coupling} to conclude that {\tt GWIPP}$(\lambda)$ is recurrent in this case. This completes the proof.
\end{proof}

\begin{proof}[Proof of Proposition \ref{thm:threshold}]
 Suppose first that
\begin{align}
 \liminf_{t\to \infty} \frac{t\lambda(t)\log 2 - \sum_{i=2}^{n-1}a_i\log^{(i)}t}{a_n\log^{(n)}t}>1 
\end{align}
 for some $n\geq 2$. (Let $n$ be minimal with this property.) Let
\begin{align}
 a:=\frac{1}{2}\left(1+\liminf_{t\to \infty} \frac{t\lambda(t)\log 2 - \sum_{i=2}^{n-1}a_i\log^{(i)}t}{a_n\log^{(n)}t} \right) > 1
\end{align}
and define
\begin{align}
 \lambda'(t) := \begin{cases} \lambda(t), & t\leq 0 \\ \frac{1}{|t|\log 2}\left(\sum_{i=2}^{n-1} a_i\log^{(i)}|t|+aa_n\log^{(n)}|t| \right), & t>0.
 \end{cases} 
\end{align}
Suppose, for contradiction, that {\tt GWIPP}$(\lambda)$ is recurrent. Since $\lambda(t)\geq \lambda'(t)$ for all $t$ large enough, Lemma \ref{lem:coupling} implies that {\tt GWIPP}$(\lambda')$ is recurrent. However, Proposition \ref{prop:thresholdregion} implies that {\tt GWIPP}$(\lambda')$ is transient. (In Proposition \ref{prop:thresholdregion} we assumed that the intensity function be symmetric, but this does not change the evaluation of the integral on the positive half-axis.) This is a contradiction, so {\tt GWIPP}$(\lambda)$ must be transient.

Now suppose the second condition holds for some $n\geq 2$. If 
\begin{align}
\lambda'(t) := \frac{1}{|t|\log 2}\left(\sum_{i=2}^{n}a_i \log ^{(i)}|t| \right),
\end{align}
then $\lambda(t)<\lambda'(t)$ for all sufficiently large $t$. By Proposition \ref{prop:thresholdregion}, {\tt GWIPP}$(\lambda')$ is recurrent, and Lemma \ref{lem:coupling} implies that {\tt GWIPP}$(\lambda)$ is recurrent.
\end{proof}

\begin{proof}[Proof of Proposition \ref{prop:slowerthanlog2}] 
For $t>0$ we have $\lambda(t)=\frac{\log f(t)}{t \log 2}$ with 
\begin{align}
 \log f(t)=\sum_{i=2}^{\infty}a_i\log^{(i)}(t)+\log g(t).
\end{align}
Because of our definition of the iterated logarithm, this implies that
\begin{align}
 f(t)=\prod_{i=1}^{\infty}\left(\max( 1,(\log^{(i)} t)^{a_{i+1}})\right) g(t).
\end{align}
Since $b_{n-1}\leq g(x)$ for any $e\uparrow \uparrow (n-1)\leq x \leq e\uparrow \uparrow n$, we obtain
\begin{align}
 \int_0^{\infty} \frac{\log f(x)}{xf(x)} \d x 
& = \sum_{n=2}^{\infty}\int_{e\uparrow \uparrow (n-1)}^{e\uparrow \uparrow n} \frac{\sum_{i=2}^{n}a_i \log ^{(i)} x+ \log g(x)}{x\prod_{i=1}^{n-1}(\log^{(i)} x)^{a_{i+1}}g(x)} \d x \\
&=\Theta(1) \sum_{n=2}^{\infty}\int_{e\uparrow \uparrow (n-1)}^{e\uparrow \uparrow n} \frac{\log ^{(2)} x}{x\prod_{i=1}^{n-1}(\log^{(i)} x)^{a_{i+1}}g(x)} \d x \\
&\leq \Theta(1) \sum_{n=2}^{\infty} \int_{e\uparrow \uparrow (n-1)}^{e\uparrow \uparrow n}\frac{1}{x\prod_{i=1}^{n-1}(\log^{(i)}x) b_{n-1}} \d x  \\
&=\Theta(1)\sum_{n=2}^{\infty}\frac{1}{b_{n-1}}\left[\log^{(n)}x\right]_{e\uparrow \uparrow (n-1)}^{e\uparrow \uparrow n} \\
&=\Theta(1)\sum_{n=2}^{\infty}\frac{1}{b_{n-1}}.
\end{align}
Using instead the bound $b_{n}\geq g(x)$ for any $e\uparrow \uparrow (n-1)\leq x \leq e\uparrow \uparrow n$, we arrive at
\begin{align}
 \sum_{n=2}^{\infty}\frac{1}{b_n}\leq \int_0^{\infty}\frac{\log f(x)}{x f(x)}\d x \leq  \Theta(1)\sum_{n=2}^{\infty}\frac{1}{b_{n-1}}.
\end{align}
Applying Lemma \ref{lem:regvar} completes the proof.
\end{proof}

\begin{remark}
Proposition \ref{thm:threshold} does not answer what happens in, for example, the regime
\begin{align}
 \liminf_{t\to \infty}\frac{t \lambda(t)\log 2}{\log \log t}<1 < \limsup_{t\to \infty}\frac{t \lambda(t) \log 2}{\log \log t}.\label{eq:between}
\end{align}
\label{rem:between}
The intensity functions $\lambda_1$ and $\lambda_2$, to be defined next, both satisfy \eqref{eq:between}, but {\tt GWIPP}($\lambda_1$) is recurrent while {\tt GWIPP}($\lambda_2)$ is transient. For $t>0$ let $\lambda_1(t)= \sum_{n=1}^{\infty} \mathbbm{1}(2^{2n}<t<2^{2n+1})$ and 
$\lambda_2(t)= \sum_{n=1}^{\infty}n \mathbbm{1}(2^n-2<t<2^n-1)$ and let $\lambda_1$ and $\lambda_2$ be symmetric around 0.
We have
\begin{align}
  0=\liminf_{t\to \infty}\frac{t \lambda_1(t)}{\log \log t}<\frac{1}{\log 2} < \limsup_{t\to \infty}\frac{t \lambda_1(t)}{\log \log t}=\infty
\end{align}
and
\begin{align}
  0=\liminf_{t\to \infty}\frac{t \lambda_2(t)}{\log \log t}<\frac{1}{\log 2} < \limsup_{t\to \infty}\frac{t \lambda_2(t)}{\log \log t}=\infty.
\end{align}
It is easy to check that 
$$\int_0^\infty \exp\left(-\int_x^{2x}\lambda_1(x)\right)\lambda_1(x)\d x=\infty,$$
which together 
with the fact that $\lambda_1$ is bounded and Corollary \ref{cor:intcond}, yields that {\tt GWIPP}($\lambda_1$) is recurrent.
From Remark \ref{peaks} we know that {\tt GWIPP}($\lambda_2)$ is transient.
\end{remark}

\end{document}